\documentclass[11pt]{article}

\usepackage{amsthm,amsmath,amsfonts,amssymb,mathrsfs}
\usepackage[numbers]{natbib}
\usepackage[colorlinks,citecolor=blue,urlcolor=blue]{hyperref}
\usepackage{graphicx,color}

\theoremstyle{plain}
\newtheorem{theorem}{Theorem}[section]
\newtheorem{corollary}[theorem]{Corollary}

\newtheorem{lemma}[theorem]{Lemma}

\theoremstyle{definition}

\theoremstyle{remark}
\newtheorem{remark}[theorem]{Remark}

\numberwithin{equation}{section}
\date{}

\title{Conditional central limit theorem for critical branching random walk \thanks{This work was supported in part by  NSFC (No.~11971062), and the National Key Research and Development Program of China (No.~2020YFA0712900).}}
\author{Wenming Hong \thanks{School of Mathematical Sciences \& Laboratory of Mathematics and Complex Systems, Beijing Normal University, Beijing 100875, P.R. China. Email: wmhong@bnu.edu.cn} ~and~ Shengli Liang \thanks{Corresponding author. School of Mathematical Sciences \& Laboratory of Mathematics and Complex Systems, Beijing Normal University, Beijing 100875, P.R. China. Email: liangshengli@mail.bnu.edu.cn}}

\begin{document}
	
	\maketitle	
	\begin{center}
		\begin{minipage}{12cm}
			\begin{center}\textbf{Abstract}\end{center}
			Consider a critical branching random walk on $\mathbb{R}$. Let $Z^{(n)}(A)$ be the number of individuals in the $n$-th generation located in $A\in \mathcal{B}(\mathbb{R})$ and $Z_{n}:=Z^{(n)}(\mathbb{R})$ denote the population of the $n$-th generation. We prove that, under some conditions, for all $x\in \mathbb{R}$, as $n\to \infty$,
			$$ \mathcal{L}\left(\frac{Z^{(n)}(-\infty, \sqrt{n} x]}{n} ~\bigg |~ Z_{n}>0\right) \Longrightarrow\mathcal{L}\left(Y(x)\right),$$
			where $\Rightarrow$ means weak convergence and $Y(x)$ is a random variable whose distribution is specified by its moments.
			
			\bigskip
			
			\mbox{}\textbf{Keywords}: Branching random walk ; critical Galton--Watson process ; reduced process ; conditional central limit theorem. \\
			\mbox{}\textbf{Mathematics Subject Classification}: Primary 60J80; secondary 60F05.
			
		\end{minipage}
	\end{center}

\section{Introduction and main result}

Consider a discrete-time \textit{branching random walk} (BRW) on the real line, which can be viewed as a natural extension of the Galton--Watson (G--W) process for the additional spatial structure. More precisely, the process starts with an initial particle positioned at the origin. It dies at time $1$ and gives birth to a random number of children who form the first generation according to the offspring distribution $p=\left\{p_{k}\right\}_{k\geq0}$, meanwhile, each of the individuals are independently positioned (with respect to their parent) according to the same probability measure $\nu$. By induction, each particle alive at generation $n-1$ dies at time $n$, and gives birth independently of all others to its own children who are in the $n$-th generation and are located in a relative position according to $p$ and $\nu$ respectively. The process goes on as described above if there are particles alive. We assume that the reproduction and displacement mechanisms are independent.

Let $Z^{(n)}(-\infty, x]$ be the number of particles of the $n$-th generation whose positions are in $(-\infty, x]$ and $Z_{n}:=Z^{(n)}(\mathbb{R})$ be the total number of particles in the $n$-th generation. $Z_{n}$ is a \textit{Galton--Watson process}, which called \textit{supercritical, critical, or subcritical} according to the mean number of offspring $m=E Z_{1}>1$, $m=1$, or $m<1$ respectively.

For a supercritical BRW (that is $m>1$), it was first conjectured by Harris (\cite{Har63}, p.75) that if $\nu$ has mean zero and variance one, then
$$\lim_{n \rightarrow \infty} \frac{Z^{(n)}(-\infty, \sqrt{n} x]}{m^{n}} =\Phi(x)W ~~\text{in probability},$$
where
$$\Phi(x)=\int_{-\infty}^{x} \frac{1}{\sqrt{2 \pi}} e^{-\frac{y^{2}}{2}} dy,~~~~W=\lim _{n \rightarrow \infty} \frac{Z_{n}}{m^{n}}.$$
Results on central limit theorem of this type have been established and extended by many authors. Stam \cite{Sta66} and Kaplan and Asmussen \cite{KA76} proved the conjecture and obtained that the convergence holds almost surely when $E(Z^2_1)<\infty$ and $E(Z_1\log^{1+\epsilon} Z_1)<\infty$, respectively. 
Klebaner \cite{Kle82} and Biggins \cite{Big90} extended these results to the branching random walk in varying environment (generation dependent). Gao, Liu and Wang \cite{GLW14} considered this central limit theorem for the BRW with random environment in time. Recently, Bansaye \cite{Ban19} extended it to the model of branching Markov chain in random environment.

In this paper, we will focus on critical case, i.e. $m=1$. The asymptotic behaviour for the BRW with critical branching mechanism is significantly different from that for supercritical case, since the corresponding G--W process becomes extinct with probability one. For instance, the asymptotic behaviour for the maximal displacement of critical BRW has been investigated by Kesten \cite{Kes95} and Lalley and Shao \cite{LS15}, the former obtained the asymptotic law for the rightmost position at time $n$ conditioned on survival up to time $\beta n$ or on extinction occurring at time $k_{n} \geqslant n$ with $k_{n}/n\to\gamma$, while the latter is devoted to studying the distribution of the rightmost position ever reached by the BRW. These asymptotic law are very different from those in the supercritical case that have been extensively studied over the past decades, see for example \cite{Aid13,Big76,Ham74,HS09,Kin75}. Here, since we will consider the central limit theorem in the critical case, it is more natural to describe this limit theorem by conditioning on survival until time $n$.

Now, we assume that
\begin{equation}\label{ass1}
	m=\mathbf{E} Z_{1}=1, \quad \sigma^{2}:=\operatorname{Var} Z_{1}\in (0,\infty).
\end{equation}
\begin{equation}\label{ass2}
	\int_{-\infty}^{\infty} x \nu(dx)=0, \quad \int_{-\infty}^{\infty} x^{2} \nu(dx)=1.
\end{equation}

We obtain the following conditional central limit theorem for the critical BRW.

\begin{theorem}\label{th1}
	Under the assumptions (\ref{ass1}), (\ref{ass2}) and that $p$ has finite moments $m_r=\sum_{j=0}^{\infty}j^rp_j$ of all orders, we have, for all $x\in \mathbb{R}$, as $n\to \infty$,
	$$\mathcal{L}\left.\left(\frac{Z^{(n)}(-\infty, \sqrt{n} x]}{n}~ \right|~ Z_{n}>0\right) \Longrightarrow\mathcal{L}\left(Y(x)\right),$$
	where $\Rightarrow$ means weak convergence and $Y(x)$ is a random variable whose distribution is determined by its moments $\mu_{r}(x)$:
	$$
	\begin{aligned}
		r=1,~~\mu_{1}(x)=&\frac{\sigma^2}{2}\Phi(x),\\
		r=2,~~\mu_{2}(x)=&\binom{2}{1}\int_{0}^{1}\mathbf{E}\left[\left(\mu_{1}^t\left(x-B_t\right)\right)^2\right]dt=\frac{\sigma^4}{2}\int_{0}^{1}\mathbf{E}\left[\Phi_{1-t}^2\left(x-B_{t}\right)\right]dt,\\
		r\geq3,~~\mu_{r}(x)=&\sum_{i=1}^{r-1}\binom{r}{i}\int_{0}^{1}\mathbf{E}\left[\mu_{i}^t\left(x-B_t\right)\mu_{r-i}^t\left(x-B_t\right)\right]dt,
	\end{aligned}
	$$
	where $B_t$ is the standard Brownian motion starting from zero, $$\Phi_{1-t}(x):=\int_{-\infty}^{x}\frac{1}{\sqrt{2 \pi(1-t)}} e^{-\frac{y^{2}}{2(1-t)}}\, dy,$$ $$\mu_{1}^t\left(x\right):=\frac{\sigma^2}{2}\Phi_{1-t}(x),~~ \mu_{2}^t\left(x\right):=\frac{\sigma^4}{2}\int_{t}^{1}\mathbf{E}\left[\Phi_{1-s}^2\left(x-B_{s}\right)\right]ds,$$
	and for $i\geq3$, $\mu_{i}^t$ is defined by recursion, i.e.,
	$$\mu_{i}^t(x):=\sum_{j=1}^{i-1}\binom{i}{j}\int_{t}^{1}\mathbf{E}\left[\mu_{j}^s\left(x-B_s\right)\mu_{i-j}^s\left(x-B_s\right)\right]ds.$$
\end{theorem}

\begin{remark}\label{rem1}
	\
	
	1. Taking $x\to\infty$, we have $\mu_{r}(\infty)=\frac{r!\sigma^{2r}}{2^r}$, Theorem \ref{th1} indicates that
	$$\mathcal{L}\left.\left(\frac{Z_{n}}{n}~ \right|~ Z_{n}>0\right) \Longrightarrow\mathcal{L}\left(Y\right),$$
	where $Y$ is a exponential random variable with mean $\sigma^{2}/{2}$. Indeed, taking $\sigma^2=2$ for simplicity, we have $\mu_{r}^t(\infty)=r!(1-t)^{r-1}$ by induction. In fact, $\mu_{1}^t(\infty)=1$, $\mu_{2}^t(\infty)=2(1-t)$. Assume that, for any $r\geq3$, $\mu_{r-1}^t(\infty)=(r-1)!(1-t)^{r-2}$. Then
	$$
	\begin{aligned}
		\mu_{r}^t(\infty)&=\sum_{i=1}^{r-1}\binom{r}{i}\int_{t}^{1}\mu_{i}^s(\infty)\mu_{r-i}^s(\infty)ds\\
		&=\sum_{i=1}^{r-1}\frac{r\cdots(r-i+1)}{i!}\int_{t}^{1}i!(1-s)^{i-1}(r-i)!(1-s)^{r-i-1}ds\\
		&=\sum_{i=1}^{r-1}\int_{t}^{1}r!(1-s)^{r-2}ds\\
		&=r!(1-t)^{r-1}.
	\end{aligned}
    $$
    It follows from the same argument as above that $\mu_{r}(\infty)=r!$.
	This is the classical theorem proved by Yaglom \cite{Yag47} under a third moment assumption (removed by Kesten et al. \cite{KNS66}). We refer to Lyons et al. \cite{LPP95}, Geiger \cite{Gei00} and Ren et al. \cite{RSS18} for elegant probabilistic proofs of Yaglom's theorem.
	
	2. It is worth noting that the limit distribution $\mathcal{L}\left(Y(x)\right)$ is related to some binary branching Brownian motion obtained by Fleischmann and Siegmund-Schultze \cite{FS78}. More specifically, the binary branching Brownian motion $\left(X_t,0\leq t<1\right)$ evolves as follows. At time zero, an initial particle moves according to a Brownian motion and splits into two at a time which is uniformly distributed on $\left(0,1\right)$. Then the two particles independently branch and move. Each particle born at time $t$ moves relatively to its position according to a Brownian motion up to a time which is uniformly distributed on $\left(t,1\right)$ and this process continues. The number of particles at time $t$ with position in $(-\infty,x]$ of the process $X_t$, rescaled by $1-t$, converges in law to $Y(x)$ as $t\to1$. 
	
	3. The assumption (\ref{ass2}) can be generalized to the case that $\nu$ is in the domain of attraction of a stable law with subtle changes on the proof of Theorem \ref{th1}. More explicitly, we assume that	there exists two positive constant sequences $\left\{a_{n}\right\}$ and $\left\{b_{n}\right\}$ such that, as $n\to\infty$,
	$$\frac{S_{n}-a_{n}}{b_{n}}\Rightarrow \eta,$$
	where $S_n$ is a sum of $n$ independent random variables with common law $\nu$, $\eta$ has a non-degenerate distribution function $G(x)$ and $G(x)$ is continuous. In this case, we have, for all $x\in \mathbb{R}$, as $n\to\infty$,
	$$ \mathcal{L}\left(\frac{Z^{(n)}(-\infty, b_{n}x +a_{n}]}{n} ~\bigg |~ Z_{n}>0\right) \Longrightarrow\mathcal{L}\left(L(x)\right),$$
	where the distribution of $L(x)$ is determined by its moments in a similar way as Theorem \ref{th1} (with the function $\Phi$ and Brownian motion being replaced by function $G$ and stable process, respectively).
	
	4. For the subcritical case, we refer to the recent work \cite{HY23}.
	\qed
\end{remark} 


\section{Reduced Galton--Watson Process}\label{section 2}

To prove our main theorem, we calculate the moments of $\left(\left.\frac{Z^{(n)}(-\infty, \sqrt{n} x]}{n}~\right| Z_{n}>0\right)$ and use the result regarding the moment convergence problem, which ensures that the limit distribution is uniquely determined by its moments. By adapting the many-to-few formula to the critical branching random walk, we can calculate explicitly its moments. The key step to formulate the many-to-few lemma in the critical case is based on the observation that to count $Z^{(n)}(A)$, the number of the particles in the $n$-th generation positioned in $A\in \mathcal{B}(\mathbb{R})$, only the ancestors at generation $k$ ($k\leqslant n$) having at least one offspring at generation $n$ make contributions. The process by which these ancestors are composed is called ``reduced branching process" in the literature.

Recall that $\{Z_n, n\geq0\}$ is a G--W process. For each fixed $n$, we denote by $Z_{k, n}$ the number of particles in the original process at time $k\leq n$ having a positive number of descendants at time $n$. That is, $\left\{Z_{k, n},0\leqslant k \leqslant n\right\}$ is the so-called \textit{reduced G--W process}, which is obtained by removing all individuals of the original process that do not have descendants at time $n$. In \cite{FS77}, Fleischmann and Siegmund-Schultze proved a functional conditional limit theorem for the reduced critical G--W process, which says that $Z_{nt,n} (0\leqslant t<1)$, given that $Z_{n}>0$, has a limit distribution specified by a Yule process with a suitable transformation of time. Also in \cite{FS78}, they obtained a conditional invariance principle establishing convergence of the reduced critical branching random walk (under contractions of time and space) to a certain binary branching Brownian motion on $[0,1)$. In our proof, we would rather make use of the behaviour of the reduced process at time $n$ than that at time $nt$.

Let us denote by $\hat{Z}^{(n)}$ a point process with law
$$\mathcal{L}\left(Z^{(n)} \mid Z_{n}>0\right)$$
and $\hat{Z}_{k, n}$ a random variable with law
$$\mathcal{L}\left(Z_{k, n} \mid Z_{n}>0\right).$$
Let $f$ be the generating function of offspring distribution $p$, that is
$$f(s):=\sum_{j=0}^{\infty} p_j s^{j},\quad s\in[0,1].$$
Let $f_{(n)}$ be the generating function of $Z_n$, by the branching property, we have
$$f_{(n)}(s)=f(f_{(n-1)}(s)),$$
with the convention that $f_{(0)}(s):=s$, $f_{(1)}(s):=f(s)$.

Fleischmann and Siegmund-Schultze \cite{FS77} investigated the reduced critical G--W process and showed that, for any $0 \leq k\leq n$, the generating function of $\hat{Z}_{k,n}$ was given by
\begin{equation}\nonumber
	\hat{f}_{(k),n}(s):=\mathbf{E}\left[s^{\hat{Z}_{k, n}}\right]=\frac{f_{(k)}\left(f_{(n-k)}(0)+s\left(1-f_{(n-k)}(0)\right)\right)-f_{(n)}(0)}{1-f_{(n)}(0)},\quad s\in[0,1].
\end{equation}
In \cite{FS77} (Proposition 1.1), they showed that $\{\hat{Z}_{k,n}; 0\leq k\leq n\}$ is a inhomogeneous G--W process, and the offspring probability generating function of a particle at time $k-1$ $\left(1\leq k\leq n\right)$ is given by
\begin{equation}\nonumber  
	\hat{f}_{k,n}(s):=\mathbf{E}\left[s^{\hat{Z}_{k,n}} \big| \hat{Z}_{k-1,n}=1\right]=\frac{f\left(f_{(n-k)}(0)+s\left(1-f_{(n-k)}(0)\right)\right)-f_{(n-k+1)}(0)}{1-f_{(n-k+1)}(0)}, s\in[0,1].
\end{equation}
Then, we get the following lemma by simple calculations.
\begin{lemma}
	Assume that the offspring distribution $p$ has finite moments of all orders, let $\hat{f}_{k,n}^{(r)}$ be the $r\left(r\geq1\right)$-order derivatives of $\hat{f}_{k,n}$, then
	\begin{equation}\label{r-derivatives}
		\hat{f}_{k,n}^{(r)}(1)=\frac{\left(1-f_{(n-k)}(0)\right)^r}{1-f_{(n-k+1)}(0)}f^{(r)}(1).
	\end{equation}
\end{lemma}


\section{Conditioned reduced spatial trees and multiple spines}\label{section 3}

In order to introduce conditioned reduced family trees of spatial branching processes and the many-to-few formula, we recall the standard Ulam--Harris--Neveu formalism for rooted, ordered, locally finite trees.

\subsection{Family trees}

We introduce the space
$$\mathcal{U}:=\{\varnothing\}\cup\bigcup_{n=1}^{\infty}\left(\mathbb{N}^*\right)^n,$$
where $\mathbb{N}^*:=\left\{1,2,\cdots\right\}$. An element in $\mathcal{U}$ is a finite sequence of integers and we think of the elements of $\mathcal{U}$ as the labels of nodes (particles) in a tree. For example, $u=236$ is the sixth child of the third child of the second child of the initial ancestor $\varnothing$. For any two elements $u$ and $v$ of $\mathcal{U}$, let $uv$ be the concatenated element, with the convention that $u\varnothing=\varnothing u=u$. For a given vertex $u=u_1\cdots u_n\in\mathcal{U}$, we denote by $|u|:=n$ its generation (with $u=\varnothing$ if $n=0$) and $\overleftarrow{u}:=u_1\cdots u_{n-1}$ its parent. 

A tree $\tau$ is a subset of $\mathcal{U}$ satisfying the following properties:
\begin{itemize}
	\item $\varnothing\in\tau$;
	\item if $u\in\tau\backslash \left\{\varnothing\right\}$, then $\overleftarrow{u}\in\tau$;
	\item for each $u\in\tau$, there exists $N_u(\tau)\in\left\{0,1,2,\cdots\right\}$ such that for every $j\in\mathbb{N}^*$, $uj\in\tau$ if and only if $1\leq j\leq N_u(\tau)$.
\end{itemize}
The integer $N_u:=N_u(\tau)$ represents the number of offspring of the particle $u\in\tau$. These rooted, ordered, locally finite trees are often called \textit{family trees}. We denote by $\mathcal{T}$ the set of all trees.


\subsection{G--W trees and G--W processes}

Let $\left(N_u,u\in\mathcal{U}\right)$ be a collection of independent random variables with distribution $p$ indexed by $\mathcal{U}$. Denote by $\mathbb{T}$ the random subset of $\mathcal{U}$ defined by
$$\mathbb{T}:=\left\{u=u_1\cdots u_n\in\mathcal{U}:n\geq0,u_j\leq N_{u_1\cdots u_{j-1}},\text{~for all~} 1\leq j\leq n\right\}.$$
Then, $\mathbb{T}$ is a \textit{G--W tree} with offspring distribution $p$ and
$$Z_n:=\#\left\{u \in \mathbb{T}:|u|=n\right\}$$ is the associated G--W process.


\subsection{Random spatial trees and branching random walks}

A \textit{spatial tree} is a tree $\tau\in\mathcal{T}$ enriched with additional spatial motion: for each $u\in\tau$, $V(u)\in\mathbb{R}$ is the position of $u$. Formally, a spatial tree is a set of the form
$$\mathbf{t}:=\left\{\left(u,V(u),u\in\tau\right)\right\}.$$
We denote by $\mathbf{T}$ the set of all spatial trees. For $\mathbf{t}\in\mathbf{T}$, we let
$$G_{n}:=\left\{u:\left(u,V(u)\right)\in\mathbf{t},|u|=n\right\}$$
be the set of all individuals at generation $n$.

A branching random walk with reproduction law $p$ and displacement distribution $\nu$ is a random spatial tree such that
\begin{itemize}
	\item $\tau$ is G--W tree with reproduction law $p$;
	\item for each $u\in\tau$, the offspring of $u$ are born at distances from $V(u)$ which are given by a family of independent random variables with common law $\nu$.
\end{itemize}
Thus, the random measure
$$Z^{(n)}(\cdot):=\sum_{u \in G_{n}} \delta_{V(u)}(\cdot)$$
defines our branching random walk. For every Borel set $A$, the number of particles located in $A$ in the $n$-th generation is defined by
$$Z^{(n)}(A):=\#\left\{u \in G_{n}: V(u) \in A\right\}.$$


\subsection{Reduced random spatial trees conditioned on $\left\{Z_{n}>0\right\}$}

Given $n\in\mathbb{N}$, the reduced random spatial tree is defined by removing all nodes of the original random tree that have no branches at time $n$. By the fact stated in Section \ref{section 2}, the reduced random spatial tree conditioned on the event $\left\{Z_{n}>0\right\}$ is a random spatial tree with inhomogeneous branching mechanism: the reproduction law of a particle $u$ with $|u|=k$ $\left(0\leq k\leq n-1\right)$, denote by $\left\{p_{l}\left(k+1,n\right)\right\}_{l\geq0}$, is determined by the generating function $\hat{f}_{k+1,n}$, the offspring's displacements are given by the same step law $\nu$.

Let $\hat{\mathbb{T}}_{n}$ denote the reduced random spatial tree conditioned on the event $\left\{Z_n>0\right\}$, which we call \textit{conditioned reduced random spatial tree}. For each $0\leq k\leq n$, we denote by
$$\hat{G}_{k,n}:=\left\{u:\left(u,V(u)\right)\in\hat{\mathbb{T}}_{n}, |u|=k\right\}$$
the set of all individuals at generation $k$ of conditioned reduced random spatial tree $\hat{\mathbb{T}}_{n}$. Then, the random measure
$$\hat{Z}^{(k,n)}(\cdot):=\sum_{u \in \hat{G}_{k,n}} \delta_{V(u)}(\cdot)$$
is the conditioned reduced BRW and the random process $\hat{Z}_{k,n}:=\#\hat{G}_{k,n}$ is the conditioned reduced G--W process. Let
\begin{equation}\nonumber
	\hat{G}_{n}:=\hat{G}_{n,n},~~\hat{Z}^{(n)}:=\hat{Z}^{(n,n)},~~\hat{Z}_{n}:=\hat{Z}_{n,n}.
\end{equation}


\subsection{Conditioned reduced random spatial trees with spines}

Given $n\in\mathbb{N}$, we now attach to $\left(\left(u,V(u)\right),u\in\hat{\mathbb{T}}_{n}\right)$ an extra random path $\left(w,V(w)\right)=\left\{w_{k},V\left(w_{k}\right)\right\}_{0\leq k\leq n}$ called spine and write $\xi_{k}:=V\left(w_{k}\right)$. For any $r\geq1$, we denote by $\hat{\tilde{\mathbb{T}}}^{[r]}_{n}$ the conditioned reduced random spatial trees with spines. To give a version of many-to-few formula for the conditioned reduced BRW, we need to define a new probability $\mathbf{Q}^{[r]}$. Under $\mathbf{Q}^{[r]}$ the particles in the conditioned reduced random spatial tree behave as follows:
\begin{itemize}
	\item At time $0$, an initial particle positioned at the origin is carrying $r$ marks $1,2,\cdots,r$. We think of each mark as distinguishing a spine and write $w^1,\cdots,w^r$ for the $r$ spines and $\xi^1,\cdots,\xi^r$ for the positions of these spines.
	\item This particle dies at time $1$ and has children whose number and displacements are determined by the $r$-th sized-biased distribution $\left\{\frac{l^rp_{l}\left(1,n\right)}{\sum_{l=0}^{\infty}l^rp_{l}\left(1,n\right)}\right\}_{l\geq0}$ (with respect to the law $\left\{p_{l}\left(1,n\right)\right\}_{l\geq0}$) and the law $\nu$, respectively. Then, the $r$ marks each choose a particle among its children to follow independently and uniformly at random.
	\item For any $1\leq k\leq n-1$, a particle with $j$ marks at generation $k$ dies at time $k+1$ and gives birth independently of all others to its children whose number and displacements are determined by the $j$-th sized-biased distribution (with respect to the law $\left\{p_{l}\left(k+1,n\right)\right\}_{l\geq0}$) and the law $\nu$, respectively, and the $j$ marks each choose a particle among its children to follow independently and uniformly at random; particles carrying no marks behave just as under $\mathbf{P}$.
\end{itemize}
Note that, under $\mathbf{Q}^{[r]}$, the particles in the spines give birth to size-biased numbers of children and move to a relative position according to the law $\nu$. The reproduction law depends on how many marks the spine particle is carrying, whereas the position does not.

Let $\mathcal{F}^{[r]}_{n}$ be the filtration containing all information about the conditioned reduced random spatial tree and the $r$ spines up to time $n$. For any $\mathcal{F}^{[r]}_{n}$-measurable $F$, we can express it as the sum
$$F=\sum_{u^1,\cdots, u^r \in \hat{G}_{n}}F\left(u^1,\cdots,u^r\right)\mathbf{1}_{\left\{w^1_{n}=u^1,\cdots,w^r_{n}=u^r\right\}},$$
where the random variable $F\left(u^1,\cdots,u^r\right)$ is $\mathcal{F}_{n}$-measurable and $\mathcal{F}_{n}$ is the natural filtration of the conditioned reduced random spatial tree. Indeed, one can show that if $A\in\mathcal{F}_{n}\cup\left\{\{w^i_{k}=u\}:u\in\mathcal{U},i\leq r, k\leq n\right\}$, then $A=\cup_{u^1,\cdots,u^r\in\mathcal{U}}\left(A(u^1,\cdots,u^r)\cap\left\{w^1_{n}=u^1,\cdots,w^r_{n}=u^r\right\}\right)$ for some collection of sets $A(u^1,\cdots,u^r)\in\mathcal{F}_{n}$. This property is retained on taking countable unions or complements, so it holds for any $A\in\mathcal{F}_{n}$. And then it is easy to check that this representation holds for any $\mathcal{F}^{[r]}_{n}$-measurable function by standard argument.


\section{Proofs}\label{section 4}

\subsection{Many-to-few formula for conditioned reduced BRW}

We assume that the assumptions in Theorem \ref{th1} hold throughout this section. For $r\geq1$, we define the $r$-th moment of the offspring distribution of a particle in the conditional reduced random spatial tree $\hat{\mathbb{T}}_{n}$ at time $k-1$ $\left(1\leq k\leq n\right)$ as
$$m_r(k):=\mathbf{E}\left.\left[\left(\hat{Z}_{k,n}\right)^r\right|\hat{Z}_{k-1,n}=1\right]=\sum_{l=0}^{\infty}l^r p_l(k,n).$$
By $\left(\ref{r-derivatives}\right)$, we have
\begin{equation}\label{m1}
	m_1(k)=\frac{1-f_{(n-k)}(0)}{1-f_{(n-k+1)}(0)}
\end{equation}
and
\begin{equation}\label{m2}
	m_2(k)=\frac{\left(1-f_{(n-k)}(0)\right)^2}{1-f_{(n-k+1)}(0)}\sigma^2+\frac{1-f_{(n-k)}(0)}{1-f_{(n-k+1)}(0)}.
\end{equation}

Now we are ready to state a version of many-to-few formula for the conditioned reduced BRW. We begin with the many-to-one and many-to-two formulas.
\begin{lemma}[Many-to-one formula]
	For any $n\in\mathbb{N}$ and any $\mathcal{F}^{[1]}_{n}$ measurable function $F$, we have
	\begin{equation}\label{many-to-one}
		\mathbf{E}\left[\sum_{u \in \hat{G}_{n}} F\left(u\right)\right]=\frac{1}{1-f_{(n)}(0)}\mathbf{E}_{\mathbf{Q}^{[1]}}\left[F\left(w^1_{n}\right) \right].
	\end{equation}
\end{lemma}

\begin{proof}
	The many-to-one formula has been used in many situations to compute the first moment, cf. \cite{Aid13,Ban19,HR17,HS09,LPP95}. In our setting, we can prove by induction that
	\begin{equation}\nonumber
		\mathbf{E}\left[\sum_{u \in \hat{G}_{n}} F\left(u\right)\right]=\mathbf{E}_{\mathbf{Q}^{[1]}}\left[F\left(w^1_{n}\right)\prod_{k=1}^{n}m_{1}(k)\right].
	\end{equation}
    Then, by $\left(\ref{m1}\right)$, we obtain $\left(\ref{many-to-one}\right)$.
\end{proof}

If we take $F\left(w^1_{n}\right)=\mathbf{1}_{\left\{V\left(w^1_{n}\right)\leq\sqrt{n}x\right\}}=\mathbf{1}_{\left\{\xi^1_{n}\leq\sqrt{n}x\right\}}$, then $F\left(u\right)=\mathbf{1}_{\left\{V(u)\leq\sqrt{n}x\right\}}$, it follows from $\left(\ref{many-to-one}\right)$ that
\begin{equation}\nonumber
	\mathbf{E}\left[\frac{\hat{Z}^{(n)}(-\infty, \sqrt{n} x]}{n}\right]=\frac{1}{n}\mathbf{E}\left[\sum_{u \in \hat{G}_{n}} \mathbf{1}_{\left\{V(u)\leq\sqrt{n}x\right\}}\right]=\frac{1}{n\left(1-f_{(n)}(0)\right)}\mathbf{Q}^{[1]}\left(\xi^1_{n}\leq\sqrt{n}x\right),
\end{equation}
where $\xi^1_{n}$ is the random walk with the following step law
\begin{equation}\label{transition of zeta}
	\mathbf{Q}^{[1]}\left(\xi^1_{1}\in\cdot\mid \xi^1_{0}=x\right)=\nu(\cdot-x).
\end{equation}
Note that
$$n\left(1-f_{(n)}(0)\right)\to\frac{2}{\sigma^{2}}$$
(cf. \cite{AN72}, Chapter I, Section 9, Theorem 1) and
$$\lim_{n\to\infty}\mathbf{Q}^{[1]}\left(\xi^1_{n}\leq\sqrt{n}x\right)=\Phi(x).$$
Then, we immediately obtain
\begin{corollary}[Limit of first moment]\label{first moment}
	\begin{equation}\label{1-moment} 
		\lim_{n\to\infty}\mathbf{E}\left[\frac{\hat{Z}^{(n)}(-\infty, \sqrt{n} x]}{n}\right]=\frac{\sigma^2}{2}\Phi(x).
	\end{equation}
\end{corollary}

To compute the second moment, we need the following many-to-two formula.
\begin{lemma}[many-to-two formula]
	For any $n\in\mathbb{N}$ and any $\mathcal{F}^{[2]}_{n}$ measurable function $F$, we have 
	\begin{equation}\label{many-to-two}
		\begin{aligned}
			&\mathbf{E}\left[\sum_{u^1, u^2 \in \hat{G}_{n}} F\left(u^1,u^2\right)\right]\\
			=&\frac{1}{1-f_{(n)}(0)}\mathbf{E}_{\mathbf{Q}^{[2]}}\left[F\left(w_{n}^1, w_{n}^1\right)\right]+\frac{\sigma^2}{1-f_{(n)}(0)}\sum_{j=1}^{n} \mathbf{E}_{\mathbf{Q}^{[2]}}\left[F\left(w_{j,n}^1, w_{j,n}^2\right)\right],
		\end{aligned}
	\end{equation}
	where $w^1_{n}$ is the spine carrying two marks, $w^1_{j,n}$ and $w^2_{j,n}$ are two spines having the same trajectory before time $j$ and splitting at time $j$.
\end{lemma}

\begin{proof}
	For a version of the many-to-two formula for the conditioned reduced BRW, we apply the many-to-two formula with martingale $\zeta \equiv 1$ in Lemma 8 of Harris and Roberts \cite{HR17}. In fact, the result in \cite{HR17} is for time-homogeneous branching random walk and the conditioned reduced BRW is time-inhomogeneous. Checking that the proof of \cite{HR17}, one only needs to replace the moment of offspring distribution $\left\{p_{l}\right\}_{l\geq0}$ with that of $\left\{p_{l}\left(k,n\right)\right\}_{l\geq0}$ at time $k-1$. 
	
	Let $T^{[2]}$ denote the first split time at which marks $1$ and $2$ are carried by different particles. Under $\mathbf{Q}^{[2]}$ the initial particle which carries marks $1$ and $2$ branches at time $1$ into $l$ particles with probability $l^2 p_l(1,n)/m_2(1)$. At such a branching event, the two marks follow the same particle with probability $1/l$. Thus
	$$\mathbf{Q}^{[2]}(T^{[2]}>1)=\frac{m_1(1)}{m_2(1)}.$$
	The distribution of $T^{[2]}$ is determined by first and second moments of offspring distribution. Indeed, under $\mathbf{Q}^{[2]}$ the particle carrying two marks branches according to $\left\{l^2 p_l(k,n)/m_2(k)\right\}$ at time $k$ and the two marks follow the same particle with probability $1/l$. Hence,
	\begin{equation}\label{tail law of T2}
		\mathbf{Q}^{[2]}(T^{[2]}>i)=\prod_{k=1}^{i}\frac{m_1(k)}{m_2(k)},\quad i\geq1,
	\end{equation}
    and
	\begin{equation}\label{law of T2}
		\mathbf{Q}^{[2]}(T^{[2]}=i)=\mathbf{Q}^{[2]}(T^{[2]}>i-1)-\mathbf{Q}^{[2]}(T^{[2]}>i)=\prod_{k=1}^{i-1}\frac{m_1(k)}{m_2(k)}\left(1-\frac{m_1(i)}{m_2(i)}\right),
	\end{equation}
	with the convention that $\prod_{k=1}^0:=1$.
	
	Since the reproduction and position of an individual are independent, writing out the many-to-two formula of \cite{HR17} and simplifying, we get
	\begin{equation}\nonumber
		\begin{aligned}
			&\mathbf{E}\left[\sum_{u^1, u^2 \in \hat{G}_{n}} F\left(u^1,u^2\right)\right]\\
			=&\mathbf{E}_{\mathbf{Q}^{[2]}}\left[F\left(w_{n}^1, w_{n}^1\right)\prod_{k=1}^{n}m_2(k)\right]\mathbf{Q}^{[2]}\left(T^{[2]}>n\right)\\
			&+\sum_{j=1}^{n}\mathbf{E}_{\mathbf{Q}^{[2]}}\left[F\left(w^1_{j,n}, w^2_{j,n}\right)\prod_{k=1}^{j}m_2(k)\prod_{k=j+1}^{n}\left(m_1(k)\right)^2\right]\mathbf{Q}^{[2]}\left(T^{[2]}=j\right),
		\end{aligned}
	\end{equation}
    where $w^1_{n}$ is the spine carrying two marks, $w^1_{j,n}$ and $w^2_{j,n}$ are two spines having the same trajectory before time $j$ and splitting at time $j$.
    
    It follows from $\left(\ref{m1}\right)$, $\left(\ref{m2}\right)$, $\left(\ref{tail law of T2}\right)$ and $\left(\ref{law of T2}\right)$ that
    \begin{equation}\nonumber
    	\prod_{k=1}^{n}m_2(k)\mathbf{Q}^{[2]}\left(T^{[2]}>n\right)
    	=\frac{1}{1-f_{(n)}(0)}
    \end{equation}
    and
    $$\prod_{k=1}^{j}m_2(k)\prod_{k=j+1}^{n}\left(m_1(k)\right)^2\mathbf{Q}^{[2]}\left(T^{[2]}=j\right)=\frac{\sigma^2}{1-f_{(n)}(0)}.$$
	Thus, we prove $\left(\ref{many-to-two}\right)$. 
\end{proof}

\begin{corollary}[Limit of second moment]\label{second moment}
	\begin{equation}\label{2-moment} 
		\lim_{n\to\infty}\mathbf{E}\left[\left(\frac{\hat{Z}^{(n)}(-\infty, \sqrt{n} x]}{n}\right)^2\right]=\frac{\sigma^4}{2}\int_{0}^{1}\mathbf{E}\left(\Phi_{1-t}^2\left(x-B_{t}\right)\right)dt,
	\end{equation}
    where $B_{t}$ is a Brownian motion and
    $$\Phi_{1-t}(x):=\int_{-\infty}^{x} \frac{1}{\sqrt{2 \pi(1-t)}} e^{-\frac{y^{2}}{2(1-t)}}\, dy.$$
\end{corollary}

\begin{proof}	
	Taking
	$$F\left(w^1_{n},w^2_{n}\right)=\mathbf{1}_{\left\{V\left(w^1_{n}\right)\leq\sqrt{n}x, V\left(w^2_{n}\right)\leq\sqrt{n}x\right\}}=\mathbf{1}_{\left\{\xi^1_{n}\leq\sqrt{n}x, \xi^2_{n}\leq\sqrt{n}x\right\}},$$
	then $$F\left(u^1,u^2\right)=\mathbf{1}_{\left\{V(u^1)\leq\sqrt{n}x,V(u^2)\leq\sqrt{n}x\right\}}.$$ 
	It follows from $\left(\ref{many-to-two}\right)$ that
	\begin{equation}\nonumber
		\begin{aligned}
			\mathbf{E}\left[\left(\frac{\hat{Z}^{(n)}(-\infty, \sqrt{n} x]}{n}\right)^2\right]=&\frac{1}{n^2}\mathbf{E}\left[\sum_{u^1,u^2\in\hat{G}_{n}}\mathbf{1}_{\left\{V(u^1)\leq\sqrt{n}x,V(u^2)\leq\sqrt{n}x\right\}}\right]\\
			=&\frac{1}{n^2\left(1-f_{(n)}(0)\right)}\mathbf{Q}^{[2]}\left(\xi_{n}^1\leq\sqrt{n}x\right)\\
			&+\frac{\sigma^2}{n^2\left(1-f_{(n)}(0)\right)}\sum_{j=1}^{n}\mathbf{Q}^{[2]}\left(\xi^1_{j,n}\leq\sqrt{n}x,\xi^2_{j,n}\leq\sqrt{n}x\right),
		\end{aligned}
	\end{equation}
    where $\xi^1_{j,n}$ and $\xi^2_{j,n}$ are two random walks having the same trajectory before time $j$ and splitting at time $j$. For the first term in last equation, taking limit as $n\to\infty$, we have
	$$\lim_{n\to\infty}\frac{1}{n^2\left(1-f_{(n)}(0)\right)}\mathbf{Q}^{[2]}\left(\xi_{n}^1\leq\sqrt{n}x\right)=0.$$
	Taking limit for the second term, we get
	\begin{equation}\label{limit of second moment}
		\begin{aligned}
			&\lim_{n\to\infty}\frac{\sigma^2}{n^2\left(1-f_{(n)}(0)\right)}\sum_{j=1}^{n}\mathbf{Q}^{[2]}\left(\xi^1_{j,n}\leq\sqrt{n}x,\xi^2_{j,n}\leq\sqrt{n}x\right)\\
			=&\lim_{n\to\infty}\frac{\sigma^2}{n^2\left(1-f_{(n)}(0)\right)}\sum_{j=1}^{n}\mathbf{Q}^{[2]}\left[\mathbf{Q}^{[2]}\left.\left(\xi^1_{j}+\tilde{\xi}_{n-j}^1\leq\sqrt{n}x,\xi^1_{j}+\tilde{\xi}_{n-j}^2\leq\sqrt{n}x~\right| \sigma\left(\xi^1_{j}\right)\right)\right]\\
			=&\lim_{n\to\infty}\frac{\sigma^2}{n^2\left(1-f_{(n)}(0)\right)}\sum_{j=1}^{n}\mathbf{Q}^{[2]}\left[\left(\mathbf{Q}^{[2]}\left.\left(\xi^1_{j}+\tilde{\xi}_{n-j}^1\leq\sqrt{n}x~\right| \sigma\left(\xi^1_{j}\right)\right)\right)^2\right]\\
			=&\lim_{n\to\infty}\frac{\sigma^2}{n\left(1-f_{(n)}(0)\right)}\frac{1}{n}\sum_{j=1}^{n}\mathbf{Q}^{[2]}\left[\left(\mathbf{Q}^{[2]}\left.\left(\frac{\xi^1_{j}}{\sqrt{n}}+\frac{\tilde{\xi}_{n-j}^1}{\sqrt{n}}\leq x ~\right| \sigma\left(\xi^1_{j}\right)\right)\right)^2\right]\\
			=&\frac{\sigma^4}{2}\int_{0}^{1}\mathbf{E}\left(\Phi_{1-t}^2\left(x-B_{t}\right)\right)dt,
		\end{aligned}
	\end{equation}
    where $\tilde{\xi}_{n-j}^1$ and $\tilde{\xi}_{n-j}^2$ are independent and identically distributed, and they are independent with $\xi^1_{j}$. The last equality follows by the central limit theorem. Indeed, for sufficiently small $\epsilon>0$, using the central limit theorem, we have that the sum
    $$\frac{1}{n}\sum_{j=1}^{n}\mathbf{Q}^{[2]}\left[\left(\mathbf{Q}^{[2]}\left.\left(\frac{\xi^1_{j}}{\sqrt{n}}+\frac{\tilde{\xi}_{n-j}^1}{\sqrt{n}}\leq x ~\right| \sigma\left(\xi^1_{j}\right)\right)\right)^2\right]$$
    is majorized for $n$ sufficiently large by the expression
    $$\frac{1}{n}\sum_{j=[\epsilon n]}^{[(1-\epsilon)n]}\mathbf{E}\left[\Phi^2\left(\left(x-\frac{\xi^1_{j}}{\sqrt{n}}\right)\left(1-\frac{j}{n}\right)^{-\frac{1}{2}}\right)\right],$$
    and this sum is the Riemmann approximating sum for the integral
    $$\int_{\epsilon}^{1-\epsilon}\mathbf{E}\left[\Phi^2\left(\left(x-B_t\right)\left(1-t\right)^{-\frac{1}{2}}\right)\right]dt=\int_{\epsilon}^{1-\epsilon}\mathbf{E}\left[\Phi^2_{1-t}\left(x-B_t\right)\right]dt.$$
    Taking $\epsilon\to0$, the last equality in display (\ref{limit of second moment}) holds.
    Thus, $\left(\ref{2-moment}\right)$ follows.
\end{proof}

Now, we state the general many-to-few formula (Lemma 8 of \cite{HR17}) in our setting for computing higher-order moments.
\begin{lemma}[many-to-few formula]
	For any $n\in\mathbb{N}$, $r\geq3$, and any $\mathcal{F}^{[r]}_{n}$ measurable function $F$, we have 
	\begin{equation}\label{many-to-few}
		\mathbf{E}\left[\sum_{u^1,\cdots, u^r \in \hat{G}_{n}} F\left(u^1,\cdots,u^r\right)\right]=\mathbf{E}_{\mathbf{Q}^{[r]}}\left[F\left(w_{n}^1,\cdots, w_{n}^r\right)\prod_{w\in skel^{(r)}(n)\backslash \left\{\varnothing\right\}}m_{D\left(\overleftarrow{w}\right)}\left(|w|\right)\right],
	\end{equation}
	where $skel^{(r)}(n)$ is the set of all particles up to time $n$ that have carried at least one mark and $D\left(w\right)$ is the number of marks carried by particle $w$.
\end{lemma}

\begin{corollary}[Limit of $r$-th moment]\label{r-th moment}
		For any $r\geq3$, let $$\mu_{r}(x):=\lim_{n\to\infty}\mathbf{E}\left[\left(\frac{\hat{Z}^{(n)}(-\infty, \sqrt{n} x]}{n}\right)^r\right],$$ then $\mu_{r}(x)$ satisfies the following recursive formula:
	\begin{equation}\label{r-moment}
		\mu_{r}(x)=\sum_{i=1}^{r-1}\binom{r}{i}\int_{0}^{1}\mathbf{E}\left[\mu_{i}^t\left(x-B_t\right)\mu_{r-i}^t\left(x-B_t\right)\right]dt,
	\end{equation}
	where $\mu_{1}^t\left(x\right):=\frac{\sigma^2}{2}\Phi_{1-t}(x)$, $\mu_{2}^t\left(x\right):=\frac{\sigma^4}{2}\int_{t}^{1}\mathbf{E}\left[\Phi_{1-s}^2\left(x-B_{s}\right)\right]ds$ and for $i\geq3$, $\mu_{i}^t$ is defined by recursion, i.e.,
	$$\mu_{i}^t(x):=\sum_{j=1}^{i-1}\binom{i}{j}\int_{t}^{1}\mathbf{E}\left[\mu_{j}^s\left(x-B_s\right)\mu_{i-j}^s\left(x-B_s\right)\right]ds.$$
	
\end{corollary}

\begin{proof}
	Let
	$$F\left(w^1_{n},\cdots,w^r_{n}\right)=\mathbf{1}_{\left\{V\left(w^1_{n}\right)\leq\sqrt{n}x,\cdots, V\left(w^r_{n}\right)\leq\sqrt{n}x\right\}}=\mathbf{1}_{\left\{\xi^1_{n}\leq\sqrt{n}x,\cdots, \xi^r_{n}\leq\sqrt{n}x\right\}},$$
	then $$F\left(u^1,\cdots,u^r\right)=\mathbf{1}_{\left\{V(u^1)\leq\sqrt{n}x,\cdots,V(u^r)\leq\sqrt{n}x\right\}}.$$ 
	It follows from $\left(\ref{many-to-few}\right)$ that
	\begin{equation}\label{many-to-few 2}
		\begin{aligned}
			\mathbf{E}\left[\left(\frac{\hat{Z}^{(n)}(-\infty, \sqrt{n} x]}{n}\right)^r\right]=&\frac{1}{n^r}\mathbf{E}\left[\sum_{u^1,\cdots,u^r\in\hat{G}_{n}}\mathbf{1}_{\left\{V(u^1)\leq\sqrt{n}x,\cdots,V(u^r)\leq\sqrt{n}x\right\}}\right]\\
			=&\frac{1}{n^r}\mathbf{E}_{\mathbf{Q}^{[r]}}\left[\mathbf{1}_{\left\{\xi^1_{n}\leq\sqrt{n}x,\cdots, \xi^r_{n}\leq\sqrt{n}x\right\}}\prod_{w\in skel^{(r)}(n)\backslash \left\{\varnothing\right\}}m_{D\left(\overleftarrow{w}\right)}\left(|w|\right)\right].
		\end{aligned}
	\end{equation}
	To compute the right side of $\left(\ref{many-to-few 2}\right)$, we should express the product in the expectation according to the marks the particles carried. We denote by $T^{[r]}$ the split time at which $r$ marks are carried by different children. Let $\left\{T^{[r]}>n\right\}$ denote the event that all $r$ marks are carried by the same particle up to time $n$, that is $r$ marks do not split, and $\left\{T^{[r]}\leq n\right\}$ denote the event that the $r$ marks have been carried by different particles at some times $0\leq j_1,\cdots,j_i\leq n$. For the latter case, more specifically, we let $T^{[r]}_1\left(1,r-1\right)$ denote the first split time at which one mark and the rest $r-1$ marks are carried by two different particles, $T^{[r]}_1\left(1,1,r-2\right)$ denote the first split time at which two of the $r$ marks are carried by two different children and the rest $r-2$ marks are carried by a different particle, $T^{[r]}_2\left(3,\left(1,r-4\right)\right)$ denote the second split time at which one of the $r-3$ marks and the rest $r-4$ marks are carried by two different particles. More generally, for $1\leq i\leq r-1$, we denote by $T^{[r]}_{i}\left(r_{1},\left(r_{2}\cdots,\left(r_{i},r-r_{1}-\cdots-r_{i}\right)\cdots\right)\right)$ the $i$-th split time at which $r_{i}$ of the $r-r_{1}-\cdots-r_{i-1}$ marks and the remaining $r-r_{1}-\cdots-r_{i}$ marks are carried by two different children.
	
	\begin{itemize}
		\item On the event that all $r$ marks do not split up to time $n$, we have
		$$
		\begin{aligned}
			&\frac{1}{n^r}\mathbf{E}_{\mathbf{Q}^{[r]}}\left[\mathbf{1}_{\left\{\xi^1_{n}\leq\sqrt{n}x,\cdots, \xi^r_{n}\leq\sqrt{n}x\right\}}\prod_{w\in skel^{(r)}(n)\backslash \left\{\varnothing\right\}}m_{D\left(\overleftarrow{w}\right)}\left(|w|\right)\mathbf{1}_{\left\{T^{[r]}>n\right\}}\right]\\
			=&\frac{1}{n^r}\mathbf{E}_{\mathbf{Q}^{[r]}}\left[\mathbf{1}_{\left\{\xi^1_{n}\leq\sqrt{n}x\right\}}\prod_{k=1}^{n}m_{r}\left(k\right)\right]\mathbf{Q}^{[r]}\left(T^{[r]}>n\right)
		\end{aligned}
		$$
		Under $\mathbf{Q}^{[r]}$, the particle which carries $r$ marks branches at time $j$ into $l$ particles with probability $l^r p_l(j,n)/m_r(j)$. At such a branching event, the $r$ marks follow the same particle with probability $1/l^{r-1}$. Thus, we have
		$$\mathbf{Q}^{[r]}\left(T^{[r]}>n\right)=\prod_{k=1}^{n}\frac{m_1(k)}{m_r(k)}.$$
		It follows that
		$$
		\begin{aligned}
			&\frac{1}{n^r}\mathbf{E}_{\mathbf{Q}^{[r]}}\left[\mathbf{1}_{\left\{\xi^1_{n}\leq\sqrt{n}x,\cdots, \xi^r_{n}\leq\sqrt{n}x\right\}}\prod_{w\in skel^{(r)}(n)\backslash \left\{\varnothing\right\}}m_{D\left(\overleftarrow{w}\right)}\left(|w|\right)\mathbf{1}_{\left\{T^{[r]}>n\right\}}\right]\\
			=&\frac{1}{n^r}\mathbf{Q}^{[r]}\left(\xi^1_{n}\leq\sqrt{n}x\right)\prod_{k=1}^{n}m_1(k)\\
			=&\frac{1}{n^r}\mathbf{Q}^{[r]}\left(\xi^1_{n}\leq\sqrt{n}x\right)\prod_{k=1}^{n}\frac{1-f_{(n-k)}(0)}{1-f_{(n-k+1)}(0)}\\
			=&\frac{1}{n^r}\mathbf{Q}^{[r]}\left(\xi^1_{n}\leq\sqrt{n}x\right)\frac{1}{1-f_{(n)}(0)},
		\end{aligned}
		$$
		Since $r\geq3$, we obtain, by taking limit as $n\to\infty$,
		$$\lim_{n\to\infty}\frac{1}{n^r}\mathbf{Q}^{[r]}\left(\xi^1_{n}\leq\sqrt{n}x\right)\frac{1}{1-f_{(n)}(0)}=0.$$
		The value on the right side of $\left(\ref{many-to-few 2}\right)$ converges to zero as $n\to\infty$ on the event that $r$ marks are carried by the same particle up to time $n$.
		
		\item On the event that $r$ marks are carried by $r$ different children at the first split time $T^{[r]}_1\left(1,\cdots,1\right)$, we have
		$$
		\begin{aligned}
			&\sum_{j=1}^{n}\frac{1}{n^r}\mathbf{E}_{\mathbf{Q}^{[r]}}\left[\mathbf{1}_{\left\{\xi^1_{n}\leq\sqrt{n}x,\cdots, \xi^r_{n}\leq\sqrt{n}x\right\}}\prod_{w\in skel^{(r)}(n)\backslash \left\{\varnothing\right\}}m_{D\left(\overleftarrow{w}\right)}\left(|w|\right)\mathbf{1}_{\left\{T^{[r]}_1\left(1,\cdots,1\right)=j\right\}}\right]\\
			=&\sum_{j=1}^{n}\frac{1}{n^r}\mathbf{E}_{\mathbf{Q}^{[r]}}\left[\mathbf{1}_{\left\{\xi^1_{j,n}\leq\sqrt{n}x,\cdots, \xi^r_{j,n}\leq\sqrt{n}x\right\}}\prod_{k=1}^{j}m_{r}\left(k\right)\prod_{k=j+1}^{n}\left(m_1(k)\right)^r\right]\\
			&\times\mathbf{Q}^{[r]}\left(T^{[r]}_1\left(1,\cdots,1\right)=j\right),
		\end{aligned}
		$$
		where $\xi^1_{j,n},\cdots,\xi^r_{j,n}$ are $r$ random walks having the same trajectory before time $j$ and splitting into $r$ different trajectories at time $j$. Under $\mathbf{Q}^{[r]}$, the particle which carries $r$ marks branches at time $j$ into $l$ particles with probability $l^r p_l(j,n)/m_r(j)$. At such a branching event, the $r$ marks follow $r$ different particles with probability $\frac{l\cdots (l-r+1)}{l^r}$. Thus, we have
		$$\mathbf{Q}^{[r]}\left(T^{[r]}_1\left(1,\cdots,1\right)=j\right)=\prod_{k=1}^{j-1}\frac{m_1(k)}{m_r(k)}\frac{\hat{f}_{j,n}^{(r)}(1)}{m_r(j)}.$$
		Then,
		$$
		\begin{aligned}
			&\sum_{j=1}^{n}\frac{1}{n^r}\mathbf{E}_{\mathbf{Q}^{[r]}}\left[\mathbf{1}_{\left\{\xi^1_{n}\leq\sqrt{n}x,\cdots, \xi^r_{n}\leq\sqrt{n}x\right\}}\prod_{w\in skel^{(r)}(n)\backslash \left\{\varnothing\right\}}m_{D\left(\overleftarrow{w}\right)}\left(|w|\right)\mathbf{1}_{\left\{T^{[r]}_1\left(1,\cdots,1\right)=j\right\}}\right]\\
			=&\sum_{j=1}^{n}\frac{1}{n^r}\mathbf{Q}^{[r]}\left(\xi^1_{j,n}\leq\sqrt{n}x,\cdots, \xi^r_{j,n}\leq\sqrt{n}x\right)\prod_{k=j+1}^{n}\left(m_1(k)\right)^r\prod_{k=1}^{j-1}m_1(k)\hat{f}_{j,n}^{(r)}(1)\\
			=&\sum_{j=1}^{n}\frac{1}{n^r}\mathbf{Q}^{[r]}\left(\xi^1_{j,n}\leq\sqrt{n}x,\cdots, \xi^r_{j,n}\leq\sqrt{n}x\right)\frac{f^{(r)}(1)}{1-f_{(n)}(0)},
		\end{aligned}
		$$
		where the last equality holds from Lemma \ref{r-derivatives} and $\left(\ref{m1}\right)$. Taking limit as $n\to\infty$, we get
		$$\lim_{n\to\infty}\sum_{j=1}^{n}\frac{1}{n^r}\mathbf{Q}^{[r]}\left(\xi^1_{j,n}\leq\sqrt{n}x,\cdots,\xi^r_{j,n}\leq\sqrt{n}x\right)\frac{f^{(r)}(1)}{1-f_{(n)}(0)}=0.$$
		That is the term on the right side of $\left(\ref{many-to-few 2}\right)$ vanishes as $n\to\infty$ on the event that $r$ marks are carried by $r$ different children at the first split time $T^{[r]}_1\left(1,\cdots,1\right)$.
		
		\item On the event that $r$ marks are carried by $r-1$ different children at the first split time $T^{[r]}_1\left(1,\cdots,1,2\right)$, we have
		$$
		\begin{aligned}
			&\sum_{j_1=1}^{n}\frac{1}{n^r}\mathbf{E}_{\mathbf{Q}^{[r]}}\left[\mathbf{1}_{\left\{\xi^1_{n}\leq\sqrt{n}x,\cdots, \xi^r_{n}\leq\sqrt{n}x\right\}}\prod_{w\in skel^{(r)}(n)\backslash \left\{\varnothing\right\}}m_{D\left(\overleftarrow{w}\right)}\left(|w|\right)\mathbf{1}_{\left\{T^{[r]}_1\left(1,\cdots,1,2\right)=j_1\right\}}\right]\\
			=&\sum_{j_1=1}^{n}\sum_{j_2=j_1+1}^{n}\frac{1}{n^r}\mathbf{Q}^{[r]}\left(\xi^1_{j_1,n}\leq\sqrt{n}x,\cdots,\xi^{r-2}_{j_1,n}\leq\sqrt{n}x,\xi^{r-1}_{j_1,j_2,n}\leq\sqrt{n}x,\xi^{r}_{j_1,j_2,n}\leq\sqrt{n}x\right)\\
			&~~~~\times\prod_{k=1}^{j_1}m_{r}\left(k\right)\prod_{k=j_1+1}^{n}\left(m_1(k)\right)^{r-2}\prod_{k=j_1+1}^{j_2}m_2(k)\prod_{k=j_2+1}^{n}\left(m_1(k)\right)^2\\
			&~~~~\times\mathbf{Q}^{[r]}\left(T^{[r]}_1\left(1,\cdots,1,2\right)=j_1,T^{[r]}_2\left(1,\cdots,1,\left(1,1\right)\right)=j_2\right)\\
			&+\sum_{j_1=1}^{n}\frac{1}{n^r}\mathbf{Q}^{[r]}\left(\xi^1_{j_1,n}\leq\sqrt{n}x,\cdots, \xi^{r-1}_{j_1,n}\leq\sqrt{n}x\right)\\
			&~~~~\times\prod_{k=1}^{j_1}m_{r}\left(k\right)\prod_{k=j_1+1}^{n}\left(m_1(k)\right)^{r-2}\prod_{k=j_1+1}^{n}m_2(k)\\
			&~~~~\times\mathbf{Q}^{[r]}\left(T^{[r]}_1\left(1,\cdots,1,2\right)=j_1,T^{[r]}_2\left(1,\cdots,1,\left(1,1\right)\right)>n\right),
		\end{aligned}
		$$
		where $\xi^{r-1}_{j_1,j_2,n},\xi^{r}_{j_1,j_2,n}$ are two random walks having the same trajectory as $\xi^1_{j_1,n}$ before time $j_1$ and splitting into two different trajectories at time $j_2$. Under $\mathbf{Q}^{[r]}$, the particle which carries $r$ marks branches at time $j_1$ into $l$ particles with probability $l^r p_l(j_1,n)/m_r(j_1)$. At such a branching event, the $r$ marks follow $r-1$ different particles with probability $\binom{r}{2}\frac{l\cdots (l-r+2)}{l^r}$. As a result, we can get
		$$
		\begin{aligned}
			&\mathbf{Q}^{[r]}\left(T^{[r]}_1\left(1,\cdots,1,2\right)=j_1,T^{[r]}_2\left(1,\cdots,1,\left(1,1\right)\right)=j_2\right)\\
			=&\prod_{k=1}^{j_1-1}\frac{m_1(k)}{m_r(k)}\frac{\binom{r}{2}\hat{f}_{j_1,n}^{(r-1)}(1)}{m_r(j_1)}\prod_{k=j_1+1}^{j_2-1}\frac{m_1(k)}{m_2(k)}\frac{\hat{f}_{j_2,n}^{(2)}(1)}{m_2(j_2)},
		\end{aligned}
		$$
		and
		$$
		\begin{aligned}
			&\mathbf{Q}^{[r]}\left(T^{[r]}_1\left(1,\cdots,1,2\right)=j_1,T^{[r]}_2\left(1,\cdots,1,\left(1,1\right)\right)>n\right)\\
			=&\prod_{k=1}^{j_1-1}\frac{m_1(k)}{m_r(k)}\frac{\binom{r}{2}\hat{f}_{j_1,n}^{(r-1)}(1)}{m_r(j_1)}\prod_{k=j_1+1}^{n}\frac{m_1(k)}{m_2(k)}.
		\end{aligned}
		$$
		Then,
		\begin{equation}\label{r carried by r-1}
			\begin{aligned}
				&\sum_{j_1=1}^{n}\frac{1}{n^r}\mathbf{E}_{\mathbf{Q}^{[r]}}\left[\mathbf{1}_{\left\{\xi^1_{n}\leq\sqrt{n}x,\cdots, \xi^r_{n}\leq\sqrt{n}x\right\}}\prod_{w\in skel^{(r)}(n)\backslash \left\{\varnothing\right\}}m_{D\left(\overleftarrow{w}\right)}\left(|w|\right)\mathbf{1}_{\left\{T^{[r]}_1\left(1,\cdots,1,2\right)=j_1\right\}}\right]\\
				=&\sum_{j_1=1}^{n}\sum_{j_2=j_1+1}^{n}\frac{1}{n^r}\mathbf{Q}^{[r]}\left(\xi^1_{j_1,n}\leq\sqrt{n}x,\cdots,\xi^{r-2}_{j_1,n}\leq\sqrt{n}x,\xi^{r-1}_{j_1,j_2,n}\leq\sqrt{n}x,\xi^{r}_{j_1,j_2,n}\leq\sqrt{n}x\right)\\
				&~~~~\times\frac{\binom{r}{2}f^{(r-1)}(1)\sigma^2}{1-f_{(n)}(0)}+\sum_{j_1=1}^{n}\frac{1}{n^r}\mathbf{Q}^{[r]}\left(\xi^1_{j_1,n}\leq\sqrt{n}x,\cdots, \xi^{r-1}_{j_1,n}\leq\sqrt{n}x\right)\frac{\binom{r}{2}f^{(r-1)}(1)}{1-f_{(n)}(0)},
			\end{aligned}
		\end{equation}
		where the last equality holds from Lemma \ref{r-derivatives} and $\left(\ref{m1}\right)$.
		\begin{itemize}
			\item When $r=3$, taking limit as $n\to\infty$, we get
			$$\lim_{n\to\infty}\sum_{j_1=1}^{n}\frac{1}{n^r}\mathbf{Q}^{[r]}\left(\xi^1_{j_1,n}\leq\sqrt{n}x,\cdots, \xi^{r-1}_{j_1,n}\leq\sqrt{n}x\right)\frac{\binom{r}{2}f^{(r-1)}(1)}{1-f_{(n)}(0)}=0.$$
			and
			$$
			\begin{aligned}
				&\lim_{n\to\infty}\sum_{j_1=1}^{n}\sum_{j_2=j_1+1}^{n}\frac{\binom{r}{2}f^{(r-1)}(1)\sigma^2}{n^r\left(1-f_{(n)}(0)\right)}\\
				&~~~~\times\mathbf{Q}^{[r]}\left(\xi^1_{j_1,n}\leq\sqrt{n}x,\cdots,\xi^{r-2}_{j_1,n}\leq\sqrt{n}x,\xi^{r-1}_{j_1,j_2,n}\leq\sqrt{n}x,\xi^{r}_{j_1,j_2,n}\leq\sqrt{n}x\right)\\
				=&\frac{3\sigma^6}{2}\lim_{n\to\infty}\sum_{j_1=1}^{n}\sum_{j_2=j_1+1}^{n}\frac{1}{n^2}\\
				&~~~~\times\mathbf{Q}^{[3]}\left(\xi^1_{j_1,n}\leq\sqrt{n}x,\xi^{2}_{j_1,j_2,n}\leq\sqrt{n}x,\xi^{3}_{j_1,j_2,n}\leq\sqrt{n}x\right)\\
				=&\frac{3\sigma^6}{2}\lim_{n\to\infty}\sum_{j_1=1}^{n}\sum_{j_2=j_1+1}^{n}\frac{1}{n^2}\\
				&~~~~\times\mathbf{Q}^{[3]}\left[\mathbf{Q}^{[3]}_{\xi^1_{j_1,j_1}}\left(\xi^1_{j_1,n}\leq\sqrt{n}x\right)\mathbf{Q}^{[3]}_{\xi^1_{j_1,j_1}}\left(\xi^{2}_{j_1,j_2,n}\leq\sqrt{n}x,\xi^{3}_{j_1,j_2,n}\leq\sqrt{n}x\right)\right]\\
				=&\frac{3\sigma^6}{2}\int_{0}^{1}\mathbf{E}\left[\Phi_{1-t_1}\left(x-B^1_{t_1}\right)\int_{t_1}^{1}\mathbf{E}_{B^1_{t_1}}\left[\Phi_{1-t_2}^2\left(x-\left(B^2_{t_2}-B^2_{t_1}\right)-B^1_{t_1}\right)\right]dt_2\right]dt_1,
			\end{aligned}
			$$
			where $B^1$ and $B^2$ are two Brownian motions such that $B^1_{s}=B^2_{s}$ for all $s\leq t_1$ and $\left\{B^1_{s+t_1}-B^1_{t_1}\right\}_{s\geq0}$ and $\left\{B^2_{s+t_1}-B^2_{t_1}\right\}_{s\geq0}$ are independent, the last equality holds by the similar calculations as $\left(\ref{limit of second moment}\right)$.
			
			So, when $r=3$, we obtain
			$$\mu_{3}(x)=\sum_{i=1}^{2}\binom{3}{i}\int_{0}^{1}\mathbf{E}\left[\mu_{i}^t\left(x-B_t\right)\mu_{3-i}^t\left(x-B_t\right)\right]dt.$$
			
			\item When $r\geq4$, it is easy to show that both terms on the right side of $\left(\ref{r carried by r-1}\right)$ tend to zero as $n\to\infty$.
		\end{itemize}
		
		\item Repeating the calculations, for any $r\geq4$, only the case that $r$ marks follow two different particles at each split time contributes to the limit of the right side of $\left(\ref{many-to-few 2}\right)$. Indeed, under $\mathbf{Q}^{[r]}$, the particle which carries $r$ marks branches at time $j_1$ into $l$ particles with probability $l^r p_l(j_1,n)/m_r(j_1)$. At such a branching event, the $r$ marks follow three different particles carrying $i_1$, $i_2$ and $r-i_1-i_2$ marks respectively, with probability $\binom{r}{i_1}\binom{r-i_1}{i_2}\frac{l(l-1)(l-2)}{l^r}$. Then
		$$
		\mathbf{Q}^{[r]}\left(T^{[r]}_1\left(i_1,i_2,r-i_1-i_2\right)=j_1\right)=\prod_{k=1}^{j_1-1}\frac{m_1\left(k\right)}{m_r\left(k\right)}\frac{\binom{r}{i_1}\binom{r-i_1}{i_2}\hat{f}^{(3)}_{j_1,n}(1)}{m_r\left(j_1\right)}.
		$$
		Thus, we have
		\begin{equation}\label{vanished term}
			\begin{aligned}
				&\lim_{n\to\infty}\frac{1}{n^r}\mathbf{E}_{\mathbf{Q}^{[r]}}\left[\mathbf{1}_{\left\{\xi^1_{n}\leq\sqrt{n}x,\cdots, \xi^r_{n}\leq\sqrt{n}x\right\}}\prod_{w\in skel^{(r)}(n)\backslash \left\{\varnothing\right\}}m_{D\left(\overleftarrow{w}\right)}\left(|w|\right)\mathbf{1}_{\left\{T^{[r]}_1\left(i_1,i_2,r-i_1-i_2\right)=j_1\right\}}\right]\\
				=&\lim_{n\to\infty}\frac{1}{n^r}\sum_{j_1=1}^{n}\mathbf{E}_{\mathbf{Q}^{[r]}}\bigg[\mathbf{1}_{\left\{\tilde{\xi}^1_{n-j_1}\leq\sqrt{n}x-\xi_{j_1},\cdots,\tilde{\xi}^{r}_{n-j_1}\leq\sqrt{n}x-\xi_{j_1}\right\}}\prod_{w\in skel^{(r)}(n-j_1)\backslash \left\{w_{j_1}\right\}}m_{D\left(\overleftarrow{w}\right)}\left(|w|\right)\bigg]\\
				&\times\prod_{k=1}^{j_1}m_{r}\left(k\right)\mathbf{Q}^{[r]}\left(T^{[r]}_1\left(i_1,i_2,r-i_1-i_2\right)=j_1\right)\\
				=&\lim_{n\to\infty}\frac{1}{n^2}\sum_{j_1=1}^{n}\mathbf{E}_{\mathbf{Q}^{[r]}}\bigg[\frac{1}{n^{i_1-1}}\mathbf{E}_{\mathbf{Q}^{[i_1]}}\bigg(\mathbf{1}_{\left\{\tilde{\xi}^{k_1}_{n-j_1}\leq\sqrt{n}x-\xi_{j_1},\cdots,\tilde{\xi}^{k_{i_1}}_{n-j_1}\leq\sqrt{n}x-\xi_{j_1}\right\}}\\
				&\times\prod_{w\in skel^{(i_1)}(n-j_1)\backslash \left\{w_{j_1}\right\}}m_{D\left(\overleftarrow{w}\right)}\left(|w|\right)~\bigg|~\sigma\left(w_{j_1},\xi_{j_1}\right)\bigg)\\
				&\times\frac{1}{n^{i_2-1}}\mathbf{E}_{\mathbf{Q}^{[i_2]}}\bigg(\mathbf{1}_{\left\{\tilde{\xi}^{k_{i_1+1}}_{n-j_1}\leq\sqrt{n}x-\xi_{j_1},\cdots,\tilde{\xi}^{k_{i_2}}_{n-j_1}\leq\sqrt{n}x-\xi_{j_1}\right\}}\\
				&\times\prod_{w\in skel^{(i_2)}(n-j_1)\backslash \left\{w_{j_1}\right\}}m_{D\left(\overleftarrow{w}\right)}\left(|w|\right)~\bigg|~\sigma\left(w_{j_1},\xi_{j_1}\right)\bigg)\\
				&\times\frac{1}{n^{r-i_1-i_2-1}}\mathbf{E}_{\mathbf{Q}^{[r-i_1-i_2]}}\bigg(\mathbf{1}_{\left\{\tilde{\xi}^{k_{i_2+1}}_{n-j_1}\leq\sqrt{n}x-\xi_{j_1},\cdots,\tilde{\xi}^{k_r}_{n-j_1}\leq\sqrt{n}x-\xi_{j_1}\right\}}\\
				&\times\prod_{w\in skel^{(r-i_1-i_2)}(n-j_1)\backslash \left\{w_{j_1}\right\}}m_{D\left(\overleftarrow{w}\right)}\left(|w|\right)~\bigg|~\sigma\left(w_{j_1},\xi_{j_1}\right)\bigg)\bigg]\\
				&\times\frac{1}{n}\prod_{k=1}^{j_1-1}m_1\left(k\right)\binom{r}{i_1}\binom{r-i_1}{i_2}\hat{f}^{(3)}_{j_1,n}(1),
			\end{aligned}
		\end{equation}
		where $skel^{(r)}(n-j_1)\backslash \left\{w_{j_1}\right\}$ is the set of all descendants of $w_{j_1}$ up to time $n$ that have carried at least one mark, $\tilde{\xi}^i_{n-j_1}:=\xi^i_{n}-\xi_{j_1}$ and $\{k_1,\cdots,k_r\}$ is a permutation of $\{1,\cdots,r\}$.
		
		The value on the right side of $\left(\ref{many-to-few 2}\right)$ converges to zero as $n\to\infty$ on the event that the $r$ marks follow three different particles carrying  $i_1$, $i_2$ and $r-i_1-i_2$ marks respectively, since one would obtain at least one extra $\frac{1}{n}$ term which makes (\ref{vanished term}) tend to zero. 
		
		Hence, we can only consider the case that $r$ marks follow two different particles at each split time.	At the first split time, the $r$ marks follow two different particles carrying $i$ and $r-i$ marks respectively, with probability $\binom{r}{i}\frac{l(l-1)}{l^r}$. So, we have
		$$
		\mathbf{Q}^{[r]}\left(T^{[r]}_1\left(i,r-i\right)=j_1\right)=\prod_{k=1}^{j_1-1}\frac{m_1\left(k\right)}{m_r\left(k\right)}\frac{\binom{r}{i}\hat{f}^{(2)}_{j_1,n}(1)}{m_r\left(j_1\right)}.
		$$
		Note that $\left\{T^{[r]}_1\left(i,r-i\right)=j_1\right\}$ is the same event as $\left\{T^{[r]}_1\left(r-i,i\right)=j_1\right\}$.
		It follows from the above argument that
		\begin{equation}\nonumber
			\begin{aligned}
				\mu_{r}(x)=&\lim_{n\to\infty}\frac{1}{n^r}\mathbf{E}_{\mathbf{Q}^{[r]}}\left[\mathbf{1}_{\left\{\xi^1_{n}\leq\sqrt{n}x,\cdots, \xi^r_{n}\leq\sqrt{n}x\right\}}\prod_{w\in skel^{(r)}(n)\backslash \left\{\varnothing\right\}}m_{D\left(\overleftarrow{w}\right)}\left(|w|\right)\right]\\
				=&\lim_{n\to\infty}\frac{1}{n^r}\sum_{i=1}^{r-1}\frac{1}{2}\sum_{j_1=1}^{n}\mathbf{E}_{\mathbf{Q}^{[r]}}\bigg[\mathbf{1}_{\left\{\tilde{\xi}^1_{n-j_1}\leq\sqrt{n}x-\xi_{j_1},\cdots,\tilde{\xi}^{r}_{n-j_1}\leq\sqrt{n}x-\xi_{j_1}\right\}}\\
				&\times\prod_{w\in skel^{(r)}(n-j_1)\backslash \left\{w_{j_1}\right\}}m_{D\left(\overleftarrow{w}\right)}\left(|w|\right)\bigg]\prod_{k=1}^{j_1}m_{r}\left(k\right)\mathbf{Q}^{[r]}\left(T^{[r]}_1\left(i,r-i\right)=j_1\right)\\
				=&\lim_{n\to\infty}\sum_{i=1}^{r-1}\binom{r}{i}\frac{1}{n}\sum_{j_1=1}^{n}\mathbf{E}_{\mathbf{Q}^{[r]}}\bigg[\frac{1}{n^{i-1}}\mathbf{E}_{\mathbf{Q}^{[i]}}\bigg(\mathbf{1}_{\left\{\tilde{\xi}^{k_1}_{n-j_1}\leq\sqrt{n}x-\xi_{j_1},\cdots,\tilde{\xi}^{k_i}_{n-j_1}\leq\sqrt{n}x-\xi_{j_1}\right\}}\\
				&\times\prod_{w\in skel^{(i)}(n-j_1)\backslash \left\{w_{j_1}\right\}}m_{D\left(\overleftarrow{w}\right)}\left(|w|\right)~\bigg|~\sigma\left(w_{j_1},\xi_{j_1}\right)\bigg)\\
				&\times\frac{1}{n^{r-i-1}}\mathbf{E}_{\mathbf{Q}^{[r-i]}}\bigg(\mathbf{1}_{\left\{\tilde{\xi}^{k_{i+1}}_{n-j_1}\leq\sqrt{n}x-\xi_{j_1},\cdots,\tilde{\xi}^{k_r}_{n-j_1}\leq\sqrt{n}x-\xi_{j_1}\right\}}\\
				&\times\prod_{w\in skel^{(r-i)}(n-j_1)\backslash \left\{w_{j_1}\right\}}m_{D\left(\overleftarrow{w}\right)}\left(|w|\right)~\bigg|~\sigma\left(w_{j_1},\xi_{j_1}\right)\bigg)\bigg]\frac{1}{n}\prod_{k=1}^{j_1-1}m_1\left(k\right)\frac{\hat{f}^{(2)}_{j_1,n}(1)}{2}\\
				=&\sum_{i=1}^{r-1}\binom{r}{i}\lim_{n\to\infty}\frac{1}{n}\sum_{j_1=1}^{n}\mathbf{E}_{\mathbf{Q}^{[r]}}\bigg[\frac{1}{n^{i-1}}\mathbf{E}_{\mathbf{Q}^{[i]}}\bigg(\mathbf{1}_{\left\{\tilde{\xi}^{k_1}_{n-j_1}\leq\sqrt{n}x-\xi_{j_1},\cdots,\tilde{\xi}^{k_i}_{n-j_1}\leq\sqrt{n}x-\xi_{j_1}\right\}}\\
				&\times\prod_{w\in skel^{(i)}(n-j_1)\backslash \left\{w_{j_1}\right\}}m_{D\left(\overleftarrow{w}\right)}\left(|w|\right)~\bigg|~\sigma\left(w_{j_1},\xi_{j_1}\right)\bigg)\\
				&\times\frac{1}{n^{r-i-1}}\mathbf{E}_{\mathbf{Q}^{[r-i]}}\bigg(\mathbf{1}_{\left\{\tilde{\xi}^{k_{i+1}}_{n-j_1}\leq\sqrt{n}x-\xi_{j_1},\cdots,\tilde{\xi}^{k_r}_{n-j_1}\leq\sqrt{n}x-\xi_{j_1}\right\}}\\
				&\times\prod_{w\in skel^{(r-i)}(n-j_1)\backslash \left\{w_{j_1}\right\}}m_{D\left(\overleftarrow{w}\right)}\left(|w|\right)~\bigg|~\sigma\left(w_{j_1},\xi_{j_1}\right)\bigg)\bigg]\frac{\left(1-f_{(n-j_1)}(0)\right)^2\sigma^2}{2n\left(1-f_{(n)}(0)\right)},
			\end{aligned}
		\end{equation}
	where the sum ``$\sum_{i=1}^{r-1}$" overcounting the case $i=1,\cdots,[r/2]$, we multiply by $\frac{1}{2}$ in the computation. By similar calculations as above, the term $\left(1-f_{(n-j_1)}(0)\right)^2$ can be dealt with $\prod_{w\in skel^{(i)}(n-j_1)\backslash \left\{w_{j_1}\right\}}m_{D\left(\overleftarrow{w}\right)}\left(|w|\right)$ and $\prod_{w\in skel^{(r-i)}(n-j_1)\backslash \left\{w_{j_1}\right\}}m_{D\left(\overleftarrow{w}\right)}\left(|w|\right)$, respectively, since particles in these two terms start from time $j_1$.
		Then, we obtain
		$$\mu_{r}(x)=\sum_{i=1}^{r-1}\binom{r}{i}\int_{0}^{1}\mathbf{E}\left[\mu_{i}^t\left(x-B_t\right)\mu_{r-i}^t\left(x-B_t\right)\right]dt.$$
	\end{itemize}
	Therefore, we complete the proof of Corollary \ref{r-th moment}.
\end{proof}

\subsection{Proof of Theorem \ref{th1}}

\begin{proof}[Proof of Theorem \ref{th1}]
	Note that
	$$\limsup_{r\to\infty}\mu^{1/2r}_{2r}(x)/2r<\infty,$$
	in fact, $\mu_{2r}(x)\leq \mu_{2r}(\infty)$ and $\mu^{1/2r}_{2r}(\infty)/2r<\infty$ by Remark \ref{rem1}.
	And then by Theorem 3.3.12 of Durrett \cite{Dur10} and Corollary \ref{first moment}, \ref{second moment} and \ref{r-th moment}, $\mathcal{L}\left.\left(\frac{Z^{(n)}(-\infty, \sqrt{n} x]}{n}~ \right| Z_{n}>0\right)$ converges weakly to the unique distribution with these moments.
\end{proof}

\end{document}